\documentclass[11pt]{article}
\usepackage[margin=2.5cm]{geometry}
\usepackage{amsmath, amsthm, amsfonts, amssymb}
\usepackage[mathscr]{euscript}
\usepackage{amsmath}
\usepackage{url}

\DeclareMathOperator{\fdim}{fdim}
\DeclareMathOperator{\At}{At}
\DeclareMathOperator{\diam}{diam}
\newcommand{\sk}[1]{{\color{black}#1}}
\usepackage[ ruled, vlined, noend]{algorithm2e} 
\usepackage{array, float}

\usepackage{graphicx, latexsym}
\usepackage{lscape}
\usepackage{multicol}
\usepackage{multirow}
\usepackage{enumerate}
\usepackage{subfigure}
\usepackage{graphicx}
\usepackage{subfig}
\usepackage{tikz}
\usepackage{cite}
\usepackage{ragged2e}
\newtheorem{thm}{Theorem}

\newtheorem{cor}{Corollary}

\newtheorem{lem}[thm]{Lemma}
\newtheorem{prop}[thm]{Proposition}

\newtheorem{ex}[thm]{Example}

\newtheorem*{rem}{Remark}

\usepackage[normalem]{ulem}
\usepackage{amsmath}
\newcommand{\stkout}[1]{\ifmmode\text{\sout{\ensuremath{#1}}}\else\sout{#1}\fi}

\newtheorem{discu}{Discussion:}

\newtheorem{conje}{Conjecture:}

\begin{document}
	\title{\textbf{Computing fault-tolerant metric dimension of graphs using their primary subgraphs} } 
	\author{\begin{tabular}{rlc}
\textbf{S. Prabhu$^{\text{a,}}$\thanks{Corresponding author: drsavariprabhu@gmail.com}\ , 
Sandi Klav\v{zar}$^{\text{b,c,d}}$,
K. Bharani Dharan$^{\text{e}}$, 
S. Radha$^{\text{e}}$}
\end{tabular}\\ 
\begin{tabular}{c}
\small $^{\text{a}}$ Department of Mathematics, Rajalakshmi Engineering College, Chennai 602105, India\\ 
\small $^{\text{b}}$ Faculty of Mathematics and Physics, University of Ljubljana, Slovenia\\
\small $^{\text{c}}$ Institute of Mathematics, Physics and Mechanics, Ljubljana, Slovenia\\
\small $^{\text{d}}$ Faculty of Natural Sciences and Mathematics, University of Maribor, Slovenia\\
\small $^{\text{e}}$ School of Advanced Sciences, Vellore Institute of Technology, Chennai 600127, India 
\end{tabular}}
\maketitle
\begin{abstract}
The metric dimension of a graph is the cardinality of a minimum resolving set, which is the set of vertices such that the distance representations of every vertex with respect to that set are unique. A fault-tolerant metric basis is a resolving set with a minimum cardinality that continues to resolve the graph even after the removal of any one of its vertices. The fault-tolerant metric dimension is the cardinality of such a fault-tolerant metric basis. In this article, we investigate the fault-tolerant metric dimension of graphs formed through the point-attaching process of primary subgraphs. This process involves connecting smaller subgraphs to specific vertices of a base graph, resulting in a more complex structure. By analyzing the distance properties and connectivity patterns, we establish explicit formulae for the fault-tolerant resolving sets of these composite graphs. Furthermore, we extend our results to specific graph products, such as rooted products. For these products, we determine the fault-tolerant metric dimension in terms of the fault-tolerant metric dimension of the primary subgraphs. Our findings demonstrate how the fault-tolerant dimension is influenced by the structural characteristics of the primary subgraphs and the attaching vertices. These results have potential applications in network design, error correction, and distributed systems, where robustness against vertex failures is crucial.
	\end{abstract}
	
	\section{Introduction}

	The metric dimension of a graph is a fundamental metric in the study of networks, identifying the smallest set of landmark vertices needed to uniquely locate all other vertices using distance measures \cite{Sl75,HaMe76}. This concept underpins numerous applications including network verification, autonomous navigation, and strategic resource placement. However, real-world networks often face disruptions due to system faults, attacks, or environmental changes \cite{YuOtHa16}. To enhance robustness, the notion of the fault-tolerant metric dimension \rm{(\textbf{FTMD})} was developed \cite{HeMoSl08}, ensuring the uniqueness of vertex identification persists even with the failure of any landmark node. This property enhances the resilience \cite{JhPi17}, dependability \cite{AlKyHu09}, and robustness of networks \cite{GhEl09}, which is crucial for domains like defense technology, sensor-based infrastructures, distributed systems, and intelligent transport networks. Calculating  \textbf{FTMD} is instrumental for constructing networks capable of withstanding failures while preserving their operational integrity, facilitating continuous monitoring and swift system adaptation during breakdowns \cite{GhEl09}. The concept of  \textbf{FTMD} was first explored in tree structures where its connection to the traditional metric dimension was examined \cite{HeMoSl08}. Later works delved into fault-tolerant resolving sets \cite{JaSaCh09}. Further investigations on characterization of graphs with respect to maximum  \textbf{FTMD} were erroneously done in ~\cite{RaHaIm19}, later it was corrected by Prabhu et al. in \cite{PrMaAr22}.  This studies extended to convex polytopes \cite{RaHaPa18},  grid graphs \cite{SiBoMa18}, and interconnection models like honeycomb and hexagonal networks were analyzed \cite{RaHaPa19}. Circulant graphs with degrees 4, 6, and 8 were thoroughly studied \cite{SeMa19,BaSaDa20}. The parameter  \textbf{FTMD} has also been investigated for convex polytopes \cite{SiHaKh19}, butterfly, Bene\v{s}, and extended honeycomb-based silicate networks \cite{PrMaAr22}.
	
	Further analysis of well-known networks such as  generalized fat-tree \cite{PrMaDa24}, biswapped network \cite{AsNaAl25}, and fractal cubic network \cite{ArKlPr23}. Algebraic graph models including zero-divisor graphs and their line graphs were also studied \cite{ShBh022}, and annihilator graphs over various ring structures have also been investigated \cite{AkMa25}.
	More recent efforts have broadened the application of  \textbf{FTMD} in specific nanotube structures with constant  \textbf{FTMD} \cite{HuMu23}. Further work includes  \textbf{FTMD} analysis on 
	arithmetic graphs \cite{SaRaCa24}, and barycentric subdivisions \cite{AhAsBa24}. The other fault-tolerant variants reported in \cite{PrJaAr25,PrJaKl25,PrJa24,BhRa25, KhHaAz25, PrArHe24} are also interesting to investigate.
		
	\section{Preliminaries}
	
Throughout this paper we denote the simple, undirected, connected graph as $G$, and the metric \( d_G: V(G) \times V(G) \rightarrow \mathbb{N}_0\) (\( d:V(G) \times V(G) \rightarrow \mathbb{N}_0\) if $G$ is understood) is defined to assign to each pair \( (x,y) \), the minimum number of edges required to connect them. Here, we denote $\{0,1,2,\ldots\}$ as $\mathbb{N}_0$. Define $r(v|X)=(d(v,x_1),d(v,x_2),\dots,d(v,x_l))$ as the representation of a vertex $v\in G$ concerning the ordered subset $X=\{x_1,x_2,\dots, x_l\}$. The subset $X$ is called a \textit{resolving set} if, for any vertices $x,y\in V(G)$, the condition $r(x|X)\ne r(y|X)$ holds true.  Let the set $X$ be classified as a {\em fault-tolerant resolving set} (\textbf{FTRS}) if, for any element $s \in X$, the set $X \setminus \{s\}$ remains a resolving set.  A \textbf{FTRS} $X$ is considered minimal if there is no other fault-tolerant resolving set $X'$ such that $X'\subset X$.  A minimal \textbf{FTRS} that possesses the least number of elements is referred to as a {\em fault-tolerant basis} (\textbf{FTB}).  The cardinality of a fault-tolerant basis is referred to as the \textbf{FTMD} of the graph $G$, coined by $\fdim(G)$.  A minimal  \textbf{FTRS} with maximum cardinality is referred to as an {\em upper fault-tolerant basis}. The cardinality of this upper \textbf{FTB} is termed the upper \textbf{FTMD} of $G$, denoted as $\fdim^{+}(G)$. The upper \textbf{FTMD} of some standard graphs are as follows. 
	\begin{itemize}
		\item for $n\ge 2$, $\fdim(P_n)=2<\fdim^{+}(P_n)=3$; 
		\item for $n\geq 5$, $\fdim(C_n)=\fdim^{+}(C_n)=3$;
		\item for $t\geq 3$, $\fdim(K_{1,t})=\fdim^{+}(K_{1,t})=t$;
		\item for $n\geq 3$, $\fdim(K_n)=\fdim^{+}(K_n)=n$. 
	\end{itemize}
	
	Let \( G \) be a resulting graph formed by merging a collection of pairwise disjoint graphs \( G_1, \dots, G_k \), $k\ge 2$, through the following process.  
	First, select two vertices one from \( G_1 \) and the other from \( G_2 \), and identify them.  Next, continue this process inductively: assume that the graphs \( G_1, \dots, G_i \) have already been included with \( i\in \{ 2, 3, \ldots, k - 1 \}\).  Choose a vertex from the current graph (possibly one of the previously identified vertices) and a vertex from \( G_{i+1} \), then identify them.  The resulting graph \( G \) has a tree-like structure, with the graphs \( G_i \) serving as its building blocks (see Figure 1).  We refer to this construction as \textit{point-attaching} of \( G_1, \dots, G_k \), where the graphs \( G_i \) are called the \textit{primary subgraphs} of \( G \) \cite{DeKl12}. Throughout this paper, the notation $[k]$ denotes the index set $\{1,2,\dots,k\}$. Here we would like to emphasize that if each of the graphs $G_i$, $i\in [k]$, is 2-connected, then the graphs $G_i$ are precisely the blocks of the graph constructed in the described way, see Figure 1 again.

	\begin{figure}[ht!]
		\label{Graph with attaching vertices}
		\centering
		\includegraphics{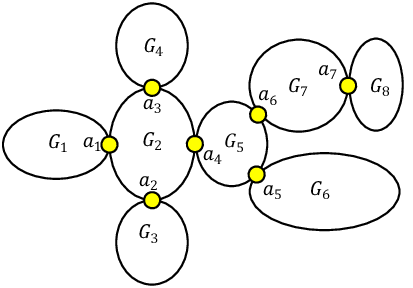}
		\caption{Graph $G$ with attaching vertices}
	\end{figure}
	
A vertex $a\in V(G)$ is said to be an attachment vertex of $G$, if $V(G_i)\cap V(G_j)=\{a\}$, for some $i$ and $j$.
The collection of all such vertices of \( G \) is denoted by \( \At(G) \), \( \At(G_i) = \At(G) \cap V(G_i) \).  Note that, $V(G_i)\cap V(G_j)=\At(G_i)\cap \At(G_j)$ and $E(G_i)\cap E(G_j)=\emptyset$, for any $i\ne j$ and $1\leq i,j\leq k$.
	
The utility of primary subgraphs, which were introduced for the first time in \cite{DeKl12}, has been demonstrated across various graph parameters. Topological indices like the Hosoya polynomial \cite{DeKl12}, atom-bond connectivity \cite{Gh22}, Graovac-Ghorbani index \cite{Gh22}, and elliptic Sombor index \cite{GhAl23}, have been efficiently computed via this approach. Additionally, parameters such as the total domination polynomial \cite{AlJa18}, distinguishing number and distinguishing index \cite{AlSo19}, and strong domination number \cite{AlGhHe23} 
	have all seen computational benefits from subgraph decomposition. Metric-based parameters in particular have shown promising results through this strategy. Certain metric dimension parameters have been successfully determined using primary subgraph frameworks \cite{KuRoYe17, RoGaBa15}, reducing computational demands.
	
	It is interesting to note that the Cactus graphs \cite{CeZe17}, rooted products of graphs \cite{MaWa23}, block graphs \cite{BeJaTa10}, corona products of graphs \cite{McMc11} are some few prominent examples of graphs formed by point-attaching process.
	
We now give a couple of basic properties of graphs formed by point-attaching. For this sake recall that a subgraph $H$ of a graph $G$ is {\em isometric}, if for every two vertices $u,v\in V(H)$ we have $d_H(u,v) = d_G(u,v)$, and that $H$ is {\em convex} if every shortest $u,v$-path from $G$ lies completely in $H$. Clearly, every convex subgraph is isometric (but not necessarily the other way around). 	
	\begin{lem}
		\label{lem:basic-properties}
		If $G$ is a graph formed by point-attaching  $G_1, \dots, G_k$, $k\geq 1$, then	the following properties hold. 
		\begin{enumerate}
			\item[(i)] $G_i$, $i\in [k]$, is a convex subgraph of $G$.
			\item[(ii)] If $u\in V(G_i)$ and $v\in V(G_j)$, where $i\ne j$, then there exist attachment vertices $x_i\in V(G_i)$ and $x_j\in V(G_j)$, such that every shortest $u,v$-path contains $x_i$ and $x_j$.
		\end{enumerate}
	\end{lem}
	
	\begin{proof}
		(i) Let $u$ and $v$ be arbitrary vertices from some primary subgraph $G_i$, where $i\in [k]$. Consider an arbitrary shortest $u,v$-path $P$ in $G$. If $P$ contains no attachment vertex, then $P$ clearly lies completely in $G_i$. Assume next that $P$ contains some attachment vertex $x$ and suppose that $x$ has a neighbor $y\notin V(G_i)$ on $P$, so that $P$ contain vertices $u, \ldots, x, y, \ldots, v$ in that order. Then by the point-attaching procedure, $P$ cannot re-enter $G_i$, that is, $P$ is not a shortest $u.v$-path. This contradiction proves that all the vertices of $P$ belong to $V(G_i)$ and we can conclude that $G_i$ is a convex subgraph of $G$. 
		
		(ii) Let $u\in V(G_i)$ and $v\in V(G_j)$, where $i\ne j$, and let $P$ be an arbitrary shortest $u,v$-path. Then by the (tree-like) point-attaching procedure, there exist a unique sequence of primary subgraphs $G_{k_1}, \ldots, G_{k_r}$, $r\ge 1$, such that $P$ sequentially contains vertices from $G_i, G_{k_1}, \ldots, G_{k_r}, G_j$. Then the claimed attachment vertex $x_i\in V(G_i)$ is the unique common vertex of $G_i$ and $G_{k_1}$, and the claimed attachment vertex $x_j\in V(G_j)$ is the unique common vertex of $G_{k_r}$ and $G_j$. 
	\end{proof}
	
	We add that in Lemma~\ref{lem:basic-properties}(ii), it is possible that $x_i = u$, or that $x_j=v$, or that $x_i=x_j$. 

	A \textit{primary end-subgraph} is a primary subgraph \( G_i \) with exactly one attachment vertex, i.e., \( |\At(G_i)| = 1 \).  It is called a \textit{primary internal subgraph} if it contains two or more attachment vertices, i.e., \( |\At(G_i)| \geq 2 \).  
	
	In this paper, we present exact expressions for the \textbf{FTMD} of graphs formed through point-attaching.  
	We apply the main result to specific graph constructions, including the rooted product and block graphs. An open neighborhood of a vertex \( v \in V(G) \) is denoted by \( N_G(v) \), and its eccentricity by \( \epsilon_G(v) \), \sk{that is, the maximum distance between $v$ and all the other vertices of $G$}. A graph \( G \) is said to be {\em even graph}~\cite{GoVe86} if, for every \( x \in V(G) \), there exist an unique \( y \in V(G) \), such that they are diametrically opposite. The hypercubes and even cycles are few examples of even graphs. Throughout the paper, definitions and concepts are introduced as they become relevant.
\section{Main results}
\label{sec:main}

To establish a foundation for our discussion, we first derive a  bound for the FTMD of $G$ formed by the point-attaching process. In this case, no specific rule governs the construction process through point-attaching, making the analysis more intricate. Since these constructions rely on the attachment process, they present additional challenges in determining their FTMD. To address this complexity, we introduce an additional parameter that directly pertains to the FTMD of graphs derived from primary subgraphs. 

An {\em attaching} \textbf{FTRS} of \( G_i \) is a subset \( \mathscr{F}_i\subset V(G_i)\setminus \At(G_i) \) such that \( \{\mathscr{F}_i\cup \At(G_i)\}\setminus \{x\}\) is a resolving set of \( G_i \) for every $x\in \mathscr{F}_i$. An {\em attaching} \textbf{FTB} of \( G_i \) is an attaching  \textbf{FTRS} of minimum cardinality, and its size is called the {\em attaching} \textbf{FTMD} of \( G_i \), denoted by \( \fdim^*(G_i, \At(G_i))\). For a given primary subgraph $G_i$, we will assume that its set $\At(G_i)$ is fixed, so in the following we may simplify the notation $\fdim^*(G_i, \At(G_i))$ to $\fdim^*(G_i)$. 

For example, if \( |\At(G_i)| = 1 \), then \( \fdim^*(G_i) \in \{\fdim(G_i), \fdim(G_i) - 1 \}\). Consider next the case $G_i = P_n$. Then \(\fdim^*(P_n) = 2\), if \(\At(P_n)\) consists of a single vertex of degree 2, while in all the other cases \(\fdim^*(P_n) = 0\). For the cycle graph \(C_n\), the value of \(\fdim^*(C_n)\) is 2 when \(\At(C_n)\) contains exactly one vertex, or when it consists of exactly two antipodal vertices and the cycle has even length. In all the other situations, \(\fdim^*(C_n) = 0\). Furthermore, $\fdim^*(K_n)=n-|\At(K_n)|$. For a complete graph $K_n$, the attaching fault-tolerant metric dimension is given by $\fdim^*(K_n) = n - |\At(K_n)|$, if $|\At(K_n)| < n-1$, and $\fdim^*(K_n) = 0$ otherwise.

\begin{prop}\label{prop1}
If $G$ is formed by the point-attaching process over $\{G_i: i \in [k]\}$, $k\geq 1$, then	\begin{equation*}
\fdim(G)\geq \sum_{i=1}^k \fdim^*(G_i).
\end{equation*}
\end{prop}

\begin{proof}
For $i\in \mathbb{N}_k$, let $\mathscr{F}_i= \mathscr{F} \cap  (V(G_i)\setminus \At(G_i))$, where $\mathscr{F}$ is a \textbf{FTB} of $G$. We claim that $\mathscr{F}_i$ is an attaching  \textbf{FTRS} of $G_i$. That is, we need to prove $\mathscr{F}_i\cup \At(G_i)\setminus\{x\}$ is a resolving set for each $x\in \mathscr{F}_i$. 

Fix $x\in \mathscr{F}_i$, and $u,v \in V(G_i)$. Since $\mathscr{F}$ is the \textbf{FTB} of $G$, the vertices $u$ and $v$ are resolved in $G$ by some vertex $y$ from $\mathscr{F}\setminus \{x\}$. Assume first that $y\in V(G_i)$. Then having in mind that $G_i$ is an isometric subgraph of $G$, vertices $u$ and $v$ are also resolved in $G_i$. Assume second that $y\in V(G_j)$ for some $j\ne i$. Then there exists a unique vertex $a\in \At(G_i)$  such that $d(y,u)=d(y,a)+d(a,u)$ and $d(y,v)=d_{G}(y,a)+d(a,v)$. Since $d(y,u) \ne d(y,v)$, it follows that $d(a,u) \ne d(a,x)$. This in turn gives $d_{G_i}(u,a) \ne d_{G_i}(x,a)$, hence $u$ and $v$ are also resolved in $G_i$ with a vertex from $\mathscr{F}_i\cup \At(G_i)\setminus\{x\}$. We have thus proved that  $\mathscr{F}_i$ is an attaching  \textbf{FTRS} of $G_i$ and therefore $|\mathscr{F}_i|\geq \fdim^*(G_i)$. This implies $$\fdim(G) = |\mathscr{F}| \ge \sum_{i=1}^k |\mathscr{F}_i|\geq \sum_{i=1}^k \fdim^*(G_i)$$
and we are done. 
\end{proof}
	
By Proposition~\ref{prop1}, for the graph illustrated in Figure~\ref{Contradiction1}, we obtain a lower bound of $4$ for the fault-tolerant metric dimension. However, the exact fault-tolerant metric dimension of this graph is $6$, showing that the bound is not tight in general. To address this gap, we introduce Condition $1$ $(\mathscr{C}_1)$. Similarly, for the graph illustrated in Figure~\ref{Contradiction2}, we obtain a lower bound of $2$ for the fault-tolerant metric dimension, whereas the exact value is $4$. This example highlights the necessity that each primary end subgraph must contain at least two fault-tolerant resolvers; therefore, we introduce Condition $2$  $(\mathscr{C}_2)$.

\begin{figure}[ht!]
	\centering
	\includegraphics{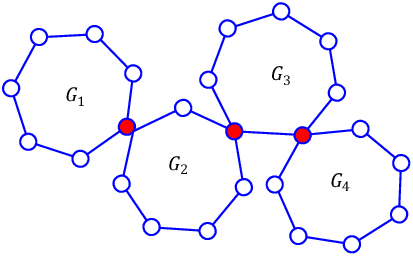}
	\caption{A graph obtained by point attachment of $G_i$, $i\in [4]$, where \sk{$G_i \cong C_7$}, $i\in [4]$.}
	\label{Contradiction1}
\end{figure}

\begin{figure}[ht!]
	\centering
	\includegraphics{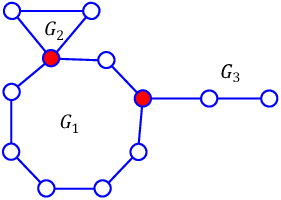}
	\caption{A graph obtained by point attachment of $G_1\cong C_8$, $G_2\cong K_3$ and $G_3\cong P_3$.}
	\label{Contradiction2}
\end{figure}

\noindent\textbf{Condition  1 $(\mathscr{C}1)$:} For each $a_1\in \At(G_i)$ and $v\in V(G_i)\setminus \At(G_i)$, there exists $a_2\in \At(G_i)$ such that $d_{G_i}(a_1,a_2)\geq d_{G_i}(v,a_2)$.
		
		\noindent\textbf{Condition 2 $(\mathscr{C}2)$:} $\At(G_i)=\{a\}$  and either $G_i\cong P_n$ and $a$ is a vertex which is not a leaf, or $G_i$ is not a path.
		
		Note that there are large class of connected graphs satifying condition \( \mathscr{C}1 \). For example, this holds when \( G_i \) meets one of the following.
		
		\begin{enumerate}
			\item $V(G_i)=\At(G_i)$.
			\item Any two vertices in $\At(G_i)$ \sk{are not} adjacent in $G_i$ and $\diam(G_i)=2$.
			\item \sk{$\epsilon_{G_i}(u_1) = \epsilon_{G_i}(u_2) = d_{G_i}(u_1,u_2)$} for any pair of distinct vertices $u_1,u_2\in \At(G_i)$. In particular, this includes all non-trivial complete graphs.
			\item $G_i$ is an even graph and if $u\in \At(G_i)$, then the vertex antipodal to $u$ also belongs to $\At(G_i)$. 
		\end{enumerate}
		 A graph \(G\) has fault-tolerant metric dimension $2$ if and only if it is a path graph~\cite{HeMoSl08}. Moreover, a set \(\{v_1, v_2\}\) forms a fault-tolerant metric basis for a path if and only if both \(v_1\) and \(v_2\) are leaves of the path. Therefore, if a graph \(G_i\) satisfies condition \(\mathscr{C}2\), then its fault-tolerant metric dimension must be at least $2$, that is, \(\fdim^*(G_i) \geq 2\).

To demonstrate that the bound of Proposition~\ref{prop1} is tight, we impose restrictions on primary subgraphs using conditions \(\mathscr{C}1\) and \(\mathscr{C}2\) as follows.  
			
		\begin{thm}\label{theorem2}
			If $G$ is formed by the point-attaching process over $\{G_i: i \in [k]\}$, \( k \geq 3 \),  
			such that each \sk{ primary internal subgraph satisfies $\mathscr{C}1$, each primary end-subgraph satisfies $\mathscr{C}2$,} 
		and no two primary end-subgraphs share a vertex, then
			\begin{equation*}
				\fdim(G)=\sum_{i=1}^k\ \fdim^*(G_i).
			\end{equation*}
		\end{thm}
		
		\begin{proof}
			By Proposition \ref{prop1}, we have \( \fdim(G) \geq \sum_{i=1}^k \fdim^*(G_i) \), hence \sk{to prove the assertion it suffices to demonstrate} that \( \fdim(G) \leq \sum_{i=1}^k \fdim^*(G_i) \).  Let \( \mathscr{F}_i \), $i\in [k]$, be an attaching \textbf{FTB} of \( G_i \). \sk{Setting
$$\mathscr{F} = \bigcup_{i=1}^k \mathscr{F}_i\,,$$
we thus need to prove that $\mathscr{F}$} forms a \textbf{FTRS} for \( G \). \sk{For this sake we must verify that for every two different vertices $x_1$ and $x_2$ of $G$, and for every vertex $y\in \mathscr{F}$, some vertex from $\mathscr{F} \setminus \{y\}$ resolves $x_1$ and $x_2$. We distinguish the following cases depending on the position of $x_1$ and $x_2$ in primary subgraphs.} 
			
\medskip\noindent
\textbf{Case 1}: $x_1,x_2\in V(G_i)$ \sk{for some $i\in [k]$.}

\medskip\noindent
\textbf{Subcase 1.1}: $\mathscr{F}_i\ne\emptyset$.\\
Since  $\mathscr{F}_i$ is non-empty and \sk{forms} an attaching fault-tolerant basis for $G_i$, there \sk{exist two vertices $u_1,u_2\in \mathscr{F}_i\cup \At(G_i)$,} such that $d(u_1,x_1)\ne d(u_1,x_2)$ and $d(u_2,x_1)\ne d(u_2,x_2)$. If $u_1,u_2\in \mathscr{F}_i$, then \sk{each of $u_1\in \mathscr{F}_i$ and $u_2\in \mathscr{F}_i$ resolves $x_1$ and $x_2$. Hence for every vertex $y\in \mathscr{F}$, the set of vertices  $\mathscr{F} \setminus \{y\}$ resolves $x_1$ and $x_2$ as required. Assume next that at least one of $u_1$ and $u_2$ belongs to $\At(G_i)$, say $u_1\in \At(G_i)$}. \sk{Consider a primary end-subgraph $G_j$, where $j\ne i$, such that among all the vertices of $G_i$, the attaching vertex $u_1$ is the vertex closest to $G_j$. Since $G_j$ obeys $\mathscr{C}2$ by a theorem's assumption, we have $|\mathscr{F}_j|\geq 2$ as argued after condition $\mathscr{C}2$ was introduced. And because $d(w,u_1)=\min\limits_{v\in V(G_i)}\{d(w,v)\}$ holds for each $w\in \mathscr{F}_j$}, \sk{every vertex $w$ from $\mathscr{F}_j$ satisfies} 
\begin{equation*}
d(w,x_1)=d(w,u_1)+d(u_1,x_1)\ne d(x_2,u_1)+d(u_1,w)=d(x_2,w).
\end{equation*}
\sk{We can conclude, having $|\mathscr{F}_j|\geq 2$ in mind, that also if at least one of $u_1$ and $u_2$ belongs to $\At(G_i)$, the set of vertices  $\mathscr{F} \setminus \{y\}$ resolves $x_1$ and $x_2$ for every vertex $y\in \mathscr{F}$.}

\medskip\noindent
\textbf{Subcase: 1.2}: $\mathscr{F}_i = \emptyset.$\\
\sk{Since $\mathscr{F}_i$ is an attaching {\bf FTB} and we have by the subcase assumption $\mathscr{F}_i = \emptyset$, the set} $\At(G_i)$ is a resolving set for $G_i$ \sk{by definition. Hence there exists a vertex $a\in \At(G_i)$ such that $d(x_1,a)\ne d(x_2,a)$. Then there exists a primary end-subgraph $G_j$, where $j\ne i$, such that} for each $w\in \mathscr{F}_{j}$ we have $d(w,a)=\min\limits_{v\in V(G_i)}\{d(w,v)\}$. Hence, for any $w\in \mathscr{F}_j$,
			\begin{equation*}
					d(x_1,w)=d(x_1,a)+d(a,w_1)\ne d(x_2,a)+d(a,w)=d(x_2,w).
			\end{equation*}            
\sk{Because we have assumed that $G_{j}$ satisfies condition $\mathscr{C}2$, we have $|\mathscr{F}_{j}|\geq 2$ as argued earlier. We can conclude that for every vertex $y\in \mathscr{F}$, the set of vertices  $\mathscr{F} \setminus \{y\}$ resolves $x_1$ and $x_2$ as required.}

\medskip\noindent
\textbf{Case 2}: $x_1\in V(G_i)$,  $x_2\in V(G_j)$, where $G_i\ne G_j$. \\ 
\sk{In this case, by the point-attaching construction of $G$, there exist vertices} $a_1\in \At(G_i)$ and $a_2\in \At(G_j)$, such that $d(x_1,x_2)=d(x_1,a_1)+d(a_1,a_2)+d(a_2,x_2).$ Note that $a_1=a_2$, whenever $\At(G_i)\cap \At(G_j)\neq \emptyset$. If $x_2=a_2=a_1$, \sk{then we actually have $x_1, x_2\in V(G_i)$ and we can apply Case 1. Similarly, if $x_1=a_1=a_2$, then $x_1, x_2\in V(G_j)$ and we can again apply Case 1.} Hence, in the rest we may assume that $x_2\ne a_1$ and $x_1\ne a_2$. \sk{We now distinguish two cases depending on whether both primary subgraphs $G_i$ and $G_j$ are internal, or at least one of them is an end-subgraph.}

\medskip\noindent
\textbf{Subcase 2.1}: $|\At(G_i)| \geq 2$ or $|\At(G_j)|\geq 2$. \\
\sk{Assume without loss of generality that $|\At(G_i)|> 1$. Then $G_i$ is a primary internal subgraph, hence by a theorem's assumption} $G_i$ obeys $\mathscr{C}1$,  \sk{that is, there exists a vertex} $c\in \At(G_i)\setminus \{a_1\}$, such that $d_{G}(a_1,c)\geq d_{G}(x_1,c)$. \sk{Consider now a primary end-subgraph $G_{\ell}$ (clearly, $\ell\ne i$), such that $d(t,c) = \min_{v\in V(G_i)}\{d(t,v)\}$ holds} for every $t\in \mathscr{F}_\ell$. As $G_\ell$ obeys $\mathscr{C}2$, \sk{we can again employ the fact that $|\mathscr{F}_\ell| \geq 2$}.  Then for every two vertices $t_1,t_2\in \mathscr{F}_\ell$, 
\begin{align*}
d(x_1,t_1) & = d(c,x_1) + d(c,t_1) \leq d(c,a_1) + d(c,t_1) \\ 
& < d(x_2,a_1) + d(c,a_1) + d(c,t_1) = d(x_2,t_1),\quad {\rm and}\\
d(x_1,t_2) & = d(c,x_1) + d(c,t_2) \leq d(c,a_1) + d(c,t_2) \\
& < d(x_2,a_1) + d(c,a_1) + d(c,t_2) = d(x_2,t_2).
\end{align*}
\sk{From the above two inequalities we can derive that for every vertex $y\in \mathscr{F}$, the set of vertices  $\mathscr{F} \setminus \{y\}$ resolves $x_1$ and $x_2$ as required.}			
			
\medskip\noindent
\textbf{Subcase 2.2}: $|\At(G_i)| =1= |\At(G_j)|$.\\ 
\sk{In this subcase both $G_i$ and $G_j$ are primary end-subgraphs, hence by a theorem's assumption they both obey condition $\mathscr{C}2$. Consequently,} $|\mathscr{F}_i| \geq 2$ and $|\mathscr{F}_j| \geq 2$. \sk{Consider arbitrary two vertices $p_1$ and $p_2$ of $\mathscr{F}_i$, and arbitrary two vertices $q_1$ and $q_2$ of $\mathscr{F}_j$.} If there exist two vertices in $\{p_1,q_1,p_2,q_2\}$ which resolve $x_1$ and $x_2$, then we are done. \sk{Using the method of contradiction,} suppose that there do not exist \sk{two vertices from $\{p_1,q_1,p_2,q_2\}$, such that they distinguish $x_1$ and $x_2$. We may without loss of generality assume} that in this case we have: 
			\begin{equation}\label{equ1}
				d(x_1,p_1)=d(x_2,a_2)+d(a_2,a_1)+d(a_1,p_1)=d(x_2,p_1),
			\end{equation}
			\begin{equation}\label{equ2}
				d(x_1,p_2)=d(x_2,a_2)+d(a_2,a_1)+d(a_1,p_2)=d(x_2,p_2),
			\end{equation}
			\begin{equation}\label{equ3}
				d(x_2,q_1)=d(x_1,a_1)+d(a_1,a_2)+d(a_2,q_1)=d(x_1,q_1).
			\end{equation}
Observe that since $\At(G_i)\cap \At(G_j) = \emptyset$, we have $a_1\ne a_2$. Moreover,
			\begin{equation}\label{equ5}
				d(x_1,p_1)\leq d(p_1,a_1)+d(a_1,x_1),
			\end{equation}
			\begin{equation}\label{equ6}
				d(x_1,p_2)\leq d(p_2,a_1)+d(a_1,x_1),
			\end{equation}
			\begin{equation}\label{equ7}
				d(x_2,q_1)\leq d(q_1,a_2)+d(a_2,x_2).
			\end{equation}
				From \eqref{equ1}, \eqref{equ2}, \eqref{equ5} and \eqref{equ6} we obtain
			\begin{equation}\label{equ9}
			d(a_1,a_2)+	d(a_2,x_2)\leq d(x_1,a_1).
			\end{equation}
			From \eqref{equ3} $\&$  \eqref{equ7}
			\begin{equation}\label{equ11}
				d(x_1,a_1)+d(a_2,a_1)\leq d(x_2,a_2).
			\end{equation}
			Adding \eqref{equ9} $\&$ \eqref{equ11}, we get
			\begin{equation*}
				2\cdot d(a_1,a_2)\leq 0,
			\end{equation*} 
			which is a contradiction. 
\sk{It implies that there exist two vertices from $\{p_1,q_1,p_2,q_2\}$ which resolve $x_1$ and $x_2$. This settles  Subcase 2.2}. 			
			
\medskip			
\sk{We have thus proved that $\mathscr{F} = \cup_{i=1}^k \mathscr{F}_i$ forms a \textbf{FTRS} for \( G \). This in turn implies that  \( \fdim(G) \leq \sum_{i=1}^k \fdim^*(G_i) \) and we are done.}
\end{proof}

Let us now illustrate Theorem~\ref{theorem2} with two examples, the first of which is a block graph. 

            \begin{ex}
				The graph $G$ illustrated in Figure~\ref{example2} satisfies conditions $(\mathscr{C}_1)$ and $(\mathscr{C}_2)$. Hence, Theorem~\ref{theorem2} can be applied to compute the fault-tolerant metric dimension of $G$ using the attaching metric dimensions of its primary subgraphs as follows:
					\begin{equation*}
						\begin{aligned}
							\fdim(G)
							&=\fdim^*(G_1)+\fdim^*(G_2)+\fdim^*(G_3)+\fdim^*(G_4)+\fdim^*(G_5)\\
							&=2+4+3+2+4\\
							&=15.
						\end{aligned}
					\end{equation*}
				\begin{figure}[H]
					\centering
					\includegraphics[width=0.5\linewidth]{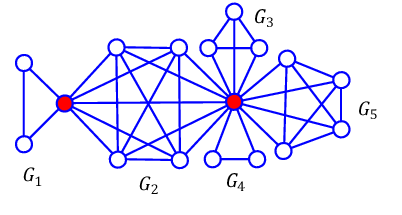}
					\caption{A block graph $G$ constructed by point attaching $G_1\cong K_3$, $G_2\cong K_6$, $G_3\cong K_4$, $G_4\cong K_3$, and $G_5\cong K_5$.}
					\label{example2}
				\end{figure}
				
			\end{ex}
            
			\begin{ex}
			The graph $G$ illustrated in Figure~\ref{Example1} satisfies conditions $(\mathscr{C}_1)$ and $(\mathscr{C}_2)$. Hence, Theorem~\ref{theorem2} can be applied to compute the fault-tolerant metric dimension of $G$ using the attaching metric dimensions of its primary subgraphs as follows:
			\begin{equation*}
				\begin{aligned}
					\fdim(G)
					&=\fdim^*(G_1)+\fdim^*(G_2)+\fdim^*(G_3)+\fdim^*(G_4)+\fdim^*(G_5)\\
					&=2+4+0+2+2\\
					&=10.
				\end{aligned}
			\end{equation*}
			\begin{figure}[H]
				\centering
				\includegraphics{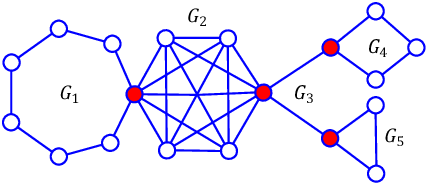}
				\caption{A graph $G$ constructed by point attaching  $G_1\cong C_7$, $G_2\cong K_6$, $G_3\cong P_3$, $G_4\cong C_4$, and $G_5\cong K_3$.}
				\label{Example1}
			\end{figure}
		
			\end{ex}
			
			The following sections focus on deriving several consequences of Theorem \ref{theorem2}.  
			Specifically, we present closed-form expressions for the \textbf{FTMD} of certain collection of graphs which are derived by point-attaching process. 
			
			\section{An extremal case}
			The investigation of the case where each minimal fault-tolerant basis of a primary subgraph is also of minimum cardinality,  
			that is, when \( \fdim(G_i) = \fdim^{+}(G_i) \) for each \( G_i \) were done in this section.  
			Let \( \mathscr{F}(G_i) \) denote the collection of all \textbf{FTB} of \( G_i \) and let
		\begin{equation*}
\theta_i=
\begin{cases}
\max\limits_{F\in\mathscr{F}(G_i)} \{|F\cap \At(G_i)|\}; & \At(G_i) \text{ is not a resolving set of } G_i,\\
\fdim(G_i); & \text{otherwise.}
\end{cases}
\end{equation*}

\begin{cor}\label{corollary3}
Let $G$ is formed by the point-attaching process over $\{G_i: i \in [k]\}$, \( k \geq 3 \), which satisfy the conditions of Theorem~\ref{theorem2}. If also \( \At(G_i) \neq V(G_i) \) and \( \fdim(G_i) = \fdim^{+}(G_i) \), $i\in [k]$, then
				\begin{equation*}
					\fdim(G)=\sum_{i=1}^k (\fdim(G_i)-\theta_i).
				\end{equation*} 
			\end{cor}
            
\begin{proof}
By the assumption, for any \sk{primary subgraph \( G_i, i\in [k] \)}, we have $\fdim(G_i)=\fdim^{+}(G_i)$. Note now that the identity
$\fdim^*(G_i)=\fdim(G_i)-\theta_i$ follows by the definition $\theta_i=\max_{F\in\mathscr{F}(G_i)} |F\cap \At(G_i)|$, because for any attaching fault-tolerant resolving set of $G_i$, the remaining vertices must necessarily be chosen from $V(G_i)\setminus \At(G_i)$, and their number is therefore $\fdim(G_i)-\theta_i$.  Hence, the result follows from Theorem~\ref{theorem2}.
\end{proof}
			
Consider the graph with five primary subgraphs from Figure~\ref{examplecor1}, where $\At(G_1)=\{a_1\}$, $\At(G_2)=\{a_1,a_2,a_3\}$, $\At(G_3)=\{a_2\}$, $\At(G_4)=\{a_3,a_4\}$, and $\At(G_5)=\{a_4\}$. Using Corollary~\ref{corollary3} we get 
\sk{
\begin{align*}
\fdim(G) & = (\fdim(G_1)-1)+(\fdim(G_2)-4)+(\fdim(G_3)-1)+ \\
& \quad\ (\fdim(G_4)-1)+(\fdim(G_5)-1)\\ 
& =3+0+2+2+4=11.
\end{align*}
}

\begin{figure}[ht!]
\centering
\includegraphics{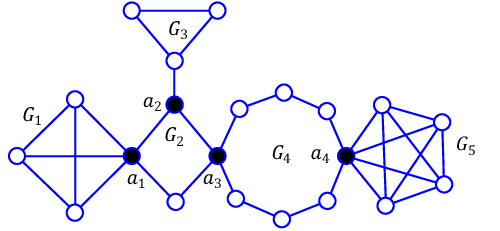}
\caption{A graph $G$ and its primary subgraphs $G_1\cong K_4$, $G_2\cong K_3$, $G_3$ being the paw graph, $G_4\cong C_8$, and $G_5\cong K_5$.}
\label{examplecor1}
\end{figure}

It is interesting to see that block graph can be derived by point-attaching process on a collection of complete graphs. Recalling that \( \fdim(K_n) = n = \fdim^+(K_n) \), we get the remark as a special case of Corollary~\ref{corollary3},  which is as follows.
	
\begin{rem}
If $G$ is a block graph formed by the point-attaching process over $\{K_{r_i}: i \in [k]\}$, where \( k \geq 3 \) and $r_i\ge 3$. If $\At(K_{r_i})\cap \At(K_{r_j})=\emptyset$ for any primary end-subgraphs $K_{r_i}$ and $K_{r_j}$, then,
\begin{equation*}
\fdim(G)=\sum_{i\in [k]\atop |\At(G_i)| < r_i-1}(r_i-|\At(K_{r_i})|).
\end{equation*}
\end{rem}

\section{Rooted products}
			
In this section, we explore a rooted product of graphs constructed through point-attaching.   
			
A rooted graph is the one with a designated vertex that is uniquely vertex labeled to distinguish it from the others.  
This distinguished vertex is referred to as the {\em root} of the graph. Let \( G \) be a vertex labeled graph of order \( n \), and let \( \mathscr{H} = \{H_i: i\in [n]\} \) be a collection of rooted graphs.  The rooted product graph denoted by \( G[\mathscr{H}] \) is formed by identifying the root of each \( H_i \) with the \( i^{\text{th}} \) vertex of \( G \) as defined in \cite{GoMc78}.  It is evident that any rooted product graph \( G[\mathscr{H}] \) can be viewed as a graph derived by point-attaching process. Using Theorem~\ref{theorem2}, we derive the following result.

\begin{cor}\label{cor5}
Let \( G \) of order \( n \geq 2 \),  
and let \( \mathscr{H} = \{H_i: i\in [n]\} \) be a collection of rooted graphs,  
each satisfying $\mathscr{C}2$, with roots \( v_1, \dots, v_n \), respectively. Then 
\begin{equation*}
\fdim(G[\mathscr{H}])=\sum_{H_i\in \mathscr{H}_1} \fdim(H_i)+\sum_{H_j\in \mathscr{H}_2}(\fdim(H_j)-1),
\end{equation*}
where \( H_i \in \mathscr{H}_1 \) if \( v_i \) is not part of any \textbf{FTB} of \( H_i \),  
and \( H_j \in \mathscr{H}_2 \) otherwise. 
\end{cor}
		
\sk{
\begin{proof}
Since in the primary internal subgraph $G$ of $G[\mathscr{H}]$, every vertex of it is an attachment vertex, $\fdim^*(G) = 0$. All the other primary subgraphs are end-subgraphs and each of them  satisfies $\mathscr{C}2$ by a theorem's assumption. Since in addition the primary internal subgraph $G$ of $G[\mathscr{H}]$ satisfies $\mathscr{C}1$ (the condition is void as all vertices of $G$ are attaching vertices), the result follows by Theorem~\ref{theorem2}.
\end{proof}		
}	
			
Next, we examine the case where the family \( \mathscr{H} \) consists of  graphs that are vertex-transitive.  
Let \( \text{Aut}(H) \) denote the automorphism group of a graph \( H \).  For any two vertices \( x_1, x_2 \in V(H) \) and any automorphism \( f \in \text{Aut}(H) \),  
the distance between the vertices is preserved. Consequently, if \( \mathscr{F} \) is a \textbf{FTB} of \( H \)  
and \( f \in \text{Aut}(H) \), then the image of the basis under the automorphism, \( f(\mathscr{F}) \),  
is also a \textbf{FTB} of \( H \). Thus, each vertex of $H$ must belong to some \textbf{FTB}.  The next remark is the direct consequence of  Corollary \ref{corollary3}. 
			
\begin{rem}
Let \( \mathscr{H} = \{H_i: i\in [n]\} \) be a collection of  vertex-transitive graphs of order at least 3, and  \( G \) of order at least 2, then  
\begin{equation*}  
\fdim(G[\mathscr{H}]) = \sum_{i=1}^n \big( \fdim(H_i) - 1 \big).  
\end{equation*}  
Further, if \( \mathscr{H} = \{K_{r_i}: i\in [n]\} \),  we have 
\begin{equation*}  
\fdim(G[\mathscr{H}]) = \sum_{i=1}^n (r_i - 1),
\end{equation*}  
and if \( \mathscr{H} = \{C_{r_i}: i\in [n]\} \),  then  
				\begin{equation*}  
					\fdim(G[\mathscr{H}]) = 2\cdot n.  
				\end{equation*}
				
			\end{rem}
A special case of rooted product graphs arises when the family \( \mathscr{H} \) consists of \( n \) isomorphic rooted graphs. Let \( V(G) = \{g_i: i \in [n]\}\), and let \( V(H) = \{h_i: i \in [n']\}\). Declare the vertex $h = h_1$ to be the root of \( H \). The {\em rooted product} \( G \circ_h H \) has the vertex set  
$$V(G \circ_h H) = V(G) \times V(H) = \big\{(g_i,h_j):\ i\in [n], j\in [n']\big\}\,,$$ 
and the edge set
$$E(G\circ_h H) = 
\big\{(g_{i},h)(g_{i'},h):\ g_ig_{i'}\in E(G)\big\} \cup 
\bigcup_{i=1}^n \big\{(g_i,h_j)(g_i,h_{j'}):\ h_jh_{j'}\in E(H)\big\}\,.$$
The following result is a special case of Corollary~\ref{corollary3}. 
						
			\begin{prop}\label{prop7}
				If \( H \) is not isomorphic to a path and $v\in V(H)$, then the following hold. 
                \begin{enumerate}
                    \item[(i)] If \( v \) is not a part of any \textbf{FTB} of \( H \), 
				then for any  \( G \) of order \( n \),  
				\begin{equation*}  
					\fdim(G \circ_v H) = n \cdot \fdim(H).  
				\end{equation*}
                    \item[(ii)] If $v$ belongs to a \textbf{FTB} of $H$, then for any $G$ of order at least 2, 
			 	\begin{equation*}
			 		\fdim(G\circ_v H)=n\cdot (\fdim(H)-1).
			 	\end{equation*}
                \end{enumerate}
			 \end{prop}

Proposition~\ref{prop7} raises the question of identifying the conditions of necessity and sufficiency  for \( v \in V(H) \) to be part of a \textbf{FTB} of \( H \).  It is straightforward to verify that a \( v \) is in  \textbf{FTB} of \( P \)  iff \( v \) is one of its leaf vertices. Building on this observation, Proposition~\ref{prop7}~(i) yields: 
			 
\begin{cor}
\label{cor:path-one-case}
Let \( v \in V(H) \) is not part of any \textbf{FTB} of \( H \), and let $G$ be a graph of order \( n \). Then $\fdim(G \circ_v H) = 2n$ iff \( H\cong P \) and $v$ is not a leaf.
\end{cor}

In view of Proposition~\ref{prop7} and Corollary~\ref{cor:path-one-case}, the remaining case to be considered is when the second factor of a rooted product graph is a path with the root as a leaf. For this specific case, the following bounds are established.
						
\begin{prop}
\label{prop:path-atttached-at-leaf}
If $G$ is of order $n\geq 2$, and $v$ is a leaf of a non-trivial path $P$, then  
	\begin{equation*}
		\fdim(G)\leq \fdim(G\circ_v P)\leq n.
	\end{equation*}
\end{prop}
\begin{proof}
$G$ appears as an induced subgraph of $G\circ_v P$. Since any  \textbf{FTRS} of $G\circ_v P$ must in particular resolve $G$, the lower bound follows. 

To establish the upper bound we claim that $V(G)\times \{v'\}$ is a  \textbf{FTRS} of $G\circ_v P$, $v'\ne v$ denotes the leaf of $P$. Let $(x',y'), (x,y)\in V(G\circ_v P)$.  If $x=x'$, then for any $u_1, u_2 \in V(G)$,  
	$$	d\!\left((u_1,v'),(x,y)\right) \ne d\!\left((u_1,v'),(x,y')\right).$$
And if $x\ne x'$, then     
$$d\!\left((x,v'),(x,y)\right) < d\!\left((x,v'),(x',y')\right) 
	\text{ and }  
	d\!\left((x',v'),(x,y)\right) >d\!\left((x',v'),(x',y')\right).$$
Hence $V(G)\times \{v'\}$ is indeed a  \textbf{FTRS} for $G\circ_v P$, hence $\fdim(G\circ_v P)\leq n$.
\end{proof}

To see that the upper bound of Proposition~\ref{prop:path-atttached-at-leaf} is sharp, consider the example from Figure~\ref{exampleforcor3}. 
            
			\begin{figure}[ht!]
				\centering
				\subfigure[]{\includegraphics{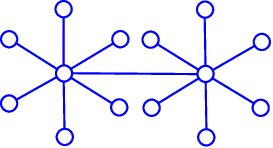}} 
				\hspace{1cm}
				\subfigure[]{\includegraphics{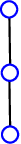}} 
				\hspace{1cm}
				\subfigure[]{\includegraphics{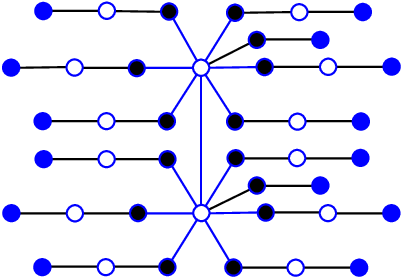}}
				\caption{(a) Graph $G$ (b) Graph $H=P_3$ (c) $G\circ_v P_3$, where the blue vertices form a fault-tolerant basis, and the black vertices are the attaching vertices.}
				\label{exampleforcor3}
			\end{figure}

		\section{Conclusion}
		
		In this paper, we have introduced an effective methodology for computing the fault-tolerant metric dimension of graphs constructed through point-attaching techniques involving primary subgraphs. By systematically analyzing the role of distance relations and connectivity among the attached substructures, we established explicit formulas for determining  \textbf{FTMD} of graphs having subgraphs that satisfying both conditions $(\mathscr{C}_1)$ and $(\mathscr{C}_2)$. This framework not only simplifies the computation for complex graphs but also enhances our understanding of how the structural properties of the primary subgraphs influence the overall fault tolerance. Explicit formulas for determining the \textbf{FTMD} of graphs whose subgraphs do not satisfy conditions $(\mathscr{C}_1)$ or $(\mathscr{C}_2)$ are under investigations.
		
		Furthermore, we extended our approach to specific graph products, including the rooted product, by expressing their fault-tolerant metric dimensions in terms of the fault-tolerant metric dimensions of their component subgraphs. This generalization highlights the versatility and applicability of our method across various graph constructions.
		
		The results presented contribute significantly to the study of fault-tolerant graph invariants and offer a modular approach for analyzing large-scale networks. Such insights are especially relevant for applications in network design, fault detection, and resilient communication systems, where maintaining unique identifiability despite node failures is essential.


\begin{thebibliography}{99}
			\bibitem{Sl75}
P.J. Slater, Leaves of trees, Congressus Numerantium \textbf{14} (1975) 549--559.

\bibitem{HaMe76}
F. Harary, R.A. Melter, On the metric dimension of a graph, Ars Combinatoria \textbf{2} (1976) 191--195.

\bibitem{YuOtHa16}
N.A.M. Yunus, M. Othman, Z.M. Hanapi, K.Y. Lun, Reliability Review of Interconnection Networks, IETE Technical Review \textbf{22}(6) (2016) 596--606.

\bibitem{HeMoSl08}
C. Hernando, M. Mora, P.J. Slater, D.R. Wood, Fault-tolerant metric dimension of graphs, In: Convexity in Discrete Structures, Ramanujan Math. Soc. Lect. Notes Ser. \textbf{5} (2008) 81--85.



\bibitem{JhPi17}
R. Jhawar, V. Piuri, Fault tolerance and resilience in cloud computing environments, in: Computer and Information Security Handbook (Third Edition), Morgan Kaufmann, Boston (2017) 155--173.

\bibitem{AlKyHu09}
M. Al-Kuwaiti, N. Kyriakopoulos, S. Hussein, A comparative analysis of network dependability, fault-tolerance, reliability, security and survivability, IEEE Communications Surveys \& Tutorials \textbf{11}(2) (2009) 106--124.

\bibitem{GhEl09}
S. Ghantasala, N.H. El-Farra, Robust diagnosis and fault-tolerant control of distributed processes over communication networks, International Journal of Adaptive Control and Signal Processing \textbf{23}(8) (2009)  699--721.

\bibitem{JaSaCh09}
I. Javaid, M. Salman, M.A. Chaudhry, S. Shokat, Fault-tolerance in resolvability, Utilitas Mathematica \textbf{80} (2009) 263.


\bibitem{RaHaIm19}
H. Raza, S. Hayat, M. Imran, X.F. Pan, Fault-tolerant resolvability and extremal structures of graphs, Mathematics \textbf{7}(1) (2019) 78.

\bibitem{PrMaAr22}
S. Prabhu, V. Manimozhi, M. Arulperumjothi, S. Klav\v{z}ar, Twin vertices in fault-tolerant metric sets and fault-tolerant metric dimension of multistage interconnection network, Applied Mathematics and Computation \textbf{420} (2022) 126897.



\bibitem{RaHaPa18} 
H. Raza, S. Hayat, X.F. Pan, On the fault-tolerant metric dimension of convex polytopes, Applied Mathematics and Computation \textbf{339} (2018) 172--185.




\bibitem{SiBoMa18}
A. Simi\'c, M. Bogdanovi\'c, Z. Maksimovi\'c, J. Milo\v{s}evi\'c, Fault-tolerant metric dimension problem: a new integer linear programming formulation and exact formula for grid graphs, Kragujevac Journal of Mathematics \textbf{42}(4) (2018) 495--503.

\bibitem{RaHaPa19}
H. Raza, S. Hayat, X.F. Pan, On the fault-tolerant metric dimension of certain interconnection networks, Journal of Applied Mathematics and Computing \textbf{60} (2018) 517--535.



\bibitem{SeMa19}
N. Seyedi, H.R. Maimani, Fault-tolerant metric dimension of circulant graphs, Facta Universitatis, Series: Mathematics and Informatics (2019) 781--788.

\bibitem{BaSaDa20}
M. Basak, L. Saha, G.K. Das, K. Tiwary, Fault-tolerant metric dimension of circulant graphs $C_n(1,2,3)$, Theoretical Computer Science \textbf{817} (2020) 66--79.


\bibitem{SiHaKh19} 
H.M.A. Siddiqui, S. Hayat, A. Khan, M. Imran, A. Razzaq, J.B. Liu, Resolvability and fault-tolerant resolvability structures of convex polytopes, Theoretical Computer Science \textbf{796} (2019) 114--128.


\bibitem{PrMaDa24}
S. Prabhu, V. Manimozhi, A. Davoodi, J.L.G. Guirao, Fault-tolerant basis of generalized fat trees and perfect binary tree derived architectures, The Journal of Supercomputing \textbf{80}(11) (2024) 15783--15798.

\bibitem{AsNaAl25} B. Assiri, M.F. Nadeem, W. Ali, A. Ahmad, Fault-tolerance in biswapped multiprocessor interconnection networks, Journal of Parallel and Distributed Computing \textbf{196} (2025) 105009.

\bibitem{ArKlPr23}
M. Arulperumjothi, S. Klav\v{z}ar, S. Prabhu, Redefining fractal cubic networks and determining their metric dimension and fault-tolerant metric dimension, Applied Mathematics and Computation \textbf{452} (2023) 128037.

\bibitem{ShBh022}
S. Sharma, V.K. Bhat, Fault-tolerant metric dimension of zero-divisor graphs of commutative rings, AKCE International Journal of Graphs and Combinatorics \textbf{19}(1) (2022) 24--30.

\bibitem{AkMa25}
M.S. Akhila, K. Manilal, Fault-tolerant metric dimension of annihilator graphs of commutative rings, Journal of Algebraic Systems \textbf{13}(1) (2025) 135--150.


\bibitem{HuMu23}
Z. Hussain, M.M. Munir, Fault-tolerance in metric dimension of boron nanotubes lattices, Frontiers in Computational Neuroscience \textbf{16} (2023) 1023585.

\bibitem{SaRaCa24}
M. Sardar, K. Rasheed, M. Cancan, M. Farahani, M. Alaeiyan, S. Patil, Fault-tolerant metric dimension of arithmetic graphs, Journal of Combinatorial Mathematics and Combinatorial Computing \textbf{122} (2024) 13--32.

\bibitem{AhAsBa24}
A. Ahmad, M.A. Asim, M.A. Ba\v{c}a, Fault-tolerant metric dimension of barycentric subdivision of Cayley graphs, Kragujevac Journal of Mathematics \textbf{48}(3) (2024) 433--439.

\bibitem{PrJaAr25} S. Prabhu, T.J. Janany, M. Arulperumjothi, I.G. Yero, Edge metric basis and its fault tolerance over certain interconnection networks, Journal of Parallel and Distributed Computing \textbf{204} (2025) 105141.

\bibitem{PrJaKl25} S. Prabhu, T.J. Janany, S. Klavžar, Metric dimensions of generalized Sierpiński graphs over squares,
Applied Mathematics and Computation \textbf{505} (2025) 129528.

\bibitem{PrJa24} S. Prabhu, T.J. Janany, Edge metric dimension of silicate networks, Communications in Combinatorics and Optimization \textbf{11}(1) (2026) 287--296. 

\bibitem{BhRa25} K.B. Dharan, S. Radha, Resolving parameters in generalized Sierpi\'nski networks over cycle of length five, European Journal of Pure and Applied Mathematics \textbf{18}(3) (2025) 6540.

\bibitem{KhHaAz25} A. Khan, S. Ali, S. Hayat, M. Azeem, Y. Zhong, 
M.A. Zahid, M.J.F. Alenazi, Fault-tolerance and unique identification of vertices and edges in a graph: 
The fault-tolerant mixed metric dimension,  Journal of Parallel and Distributed Computing 197 (2025) 105024.

\sk{
\bibitem{PrArHe24} 
S. Prabhu, A.K. Arulmozhi, M.A. Henning, M. Arulperumjothi, 
Fault-tolerant resolving power domination of fractal cubic network,
Journal of Parallel and Distributed Computing \textbf{211} (2026) 105243.
}

\bibitem{DeKl12}
E. Deutsch, S. Klav\v{z}ar, Computing Hosoya polynomials of graphs from primary subgraphs, MATCH Communications in Mathematical and in Computer Chemistry \textbf{70} (2013) 627--644.

\bibitem{Gh22}
N. Ghanbari, On the Graovac-Ghorbani and atom-bond connectivity indices of graphs from primary subgraphs, Iranian Journal of Mathematical Chemistry \textbf{13}(1) (2022) 45--72.

\bibitem{GhAl23}
N. Ghanbari, S. Alikhani, Elliptic Sombor index of graphs from primary subgraphs, Analytical and Numerical Solutions for Nonlinear Equations \textbf{8}(1) (2023) 99--109.


\bibitem{AlJa18}
S. Alikhani, N. Jafari, Total domination polynomial of graphs from primary subgraphs, Journal of Algebraic Systems \textbf{5}(2) (2018) 127--138.

\bibitem{AlSo19}
S. Alikhani, S. Soltani, The distinguishing number and the distinguishing index of graphs from primary subgraphs, Iranian Journal of Mathematical Chemistry \textbf{10}(3) (2019) 223--240.

\bibitem{AlGhHe23}
S. Alikhani, N. Ghanbari, M.A. Henning, Strong domination number of graphs from primary subgraphs, \url{arXiv:2306.01608} (2023).



\bibitem{KuRoYe17}
D. Kuziak, J.A. Rodr\'iguex-Vela\'azquez, I.G. Yero, Computing the metric dimension of a graph from primary subgraphs, Discussiones Mathematicae Graph Theory \textbf{37}(1) (2017) 273--293.

\bibitem{RoGaBa15}
J.A. Rodr\'iguex-Vela\'azquez, C.G. G\'omez, G.A. Barrag\'an-Ram\'irez, Computing the local metric dimension of a graph from the local metric dimension of primary subgraphs, International Journal of Computer Mathematics \textbf{92}(4) (2015)  686--693.

\bibitem{CeZe17} M. Čevnik, J. Žerovnik, Broadcasting on cactus graphs, Journal of Combinatorial Optimization \textbf{33}(1) (2017) 292--316.

\bibitem{MaWa23} L. Mao, W. Wang, Generalized spectral characterization of rooted product graphs, Linear and Multilinear Algebra \textbf{71}(14) (2023) 2310--2324.

\bibitem{BeJaTa10} A. Behtoei, M. Jannesari, B. Taeri, A characterization of block graphs, Discrete Applied Mathematics \textbf{158}(3) (2010) 219--221.

\bibitem{McMc11} C. McLeman, E. McNicholas, Spectra of coronae, Linear Algebra and its Applications \textbf{435}(5) (2011) 998--1007.


\bibitem{GoVe86}
F. G\"{o}bel, H.J. Veldman, Even graphs, Journal of Graph Theory \textbf{10}(2) (1986) 225--239.

\bibitem{GoMc78} C.D. Godsil, B.D. McKay, A new graph product and its spectrum, Bulletin of the Australian Mathematical Society \textbf{18} (1) (1978) 21--28.			
		\end{thebibliography}
\end{document}